\documentclass[12pt,letterpaper,twoside]{article}

\usepackage{amsmath,amssymb}
\usepackage{amsthm}
\usepackage{stmaryrd}
\usepackage{fancyhdr}
\usepackage{times}
\usepackage{mathrsfs} 
\usepackage{titlesec}
\usepackage[normalem]{ulem}
\usepackage[pdftex,dvips]{geometry}

\newtheoremstyle{fact}
     {\topsep}
     {\topsep}
     {\slshape}
     {}
     {\bfseries}
     {}
     { }
     {\thmname{#1}\thmnumber{ #2.}\thmnote{ \rm (#3)}}

\newtheoremstyle{mylabel}
     {\topsep}
     {\topsep}
     {\itshape}
     {}
     {\bfseries}
     {}
     { }
     {\thmname{#1}\thmnote{ #3}.}

\newtheorem{theorem}{Theorem}[section]
\newtheorem{Ltheorem}{Theorem}

\newtheorem*{theorem*}{Theorem} 
\newtheorem{lemma}[theorem]{Lemma}
\newtheorem{proposition}[theorem]{Proposition}
\newtheorem{corollary}[theorem]{Corollary}

\theoremstyle{mylabel}
\newtheorem*{Ltheorem*}{Theorem}

\theoremstyle{definition}
\newtheorem{definition}[theorem]{Definition}
\newtheorem{remark}[theorem]{Remark}

\newtheorem*{remark*}{Remark}

\newtheorem*{question*}{Question}
\newtheorem*{examples*}{Examples}  
\newtheorem{example}[theorem]{Example}
\newtheorem{examples}[theorem]{Examples}
\newtheorem*{example*}{Example}

\newtheorem*{convention*}{Convention}

\theoremstyle{fact}

\newtheorem{ftheorem}[theorem]{Theorem}
\newtheorem{flemma}[theorem]{Lemma}

\newenvironment{myromanlist}[1][enumi]{\begin{list}{{\rm (\roman{#1})}}
{\usecounter{#1}\setlength{\labelwidth}{25pt}\setlength{\topsep}{-6pt}
\setlength{\itemsep}{-4pt} \setlength{\leftmargin}{25pt}}}{\end{list}}

\newenvironment{myalphlist}[1][enumi]{\begin{list}{{\rm (\alph{#1})}}
{\usecounter{#1}\setlength{\labelwidth}{25pt}\setlength{\topsep}{-6pt}
\setlength{\itemsep}{-4pt} \setlength{\leftmargin}{25pt}}}{\end{list}}

\newenvironment{mynumlist}[1][enumii]{\begin{list}{{\rm (\arabic{#1})}}
{\usecounter{#1}\setlength{\labelwidth}{25pt}\setlength{\topsep}{-6pt}
\setlength{\itemsep}{-4pt} \setlength{\leftmargin}{25pt}}}{\end{list}}

\def\proofont{\fontseries{bx}\fontshape{sc}\selectfont}
\def\proofname{Proof.}

\newcommand{\pcite}[2]{{\cite[#1]{#2}}}

\newcommand{\Note}[1]{}

\makeatletter
\renewenvironment{proof}[1][\proofname]{\par
  \normalfont
  \topsep6\p@\@plus6\p@ \trivlist
  \item[\hskip\labelsep\noindent\proofont #1]\ignorespaces
}{%
  \qed\endtrivlist
}
\makeatother

\titlelabel{\thetitle.\ }
\titleformat*{\section}{\normalsize\bfseries\centering}
\titleformat*{\subsection}{\normalsize\bfseries\itshape}

\author{D. Dikranjan\thanks{The first author acknowledges the financial 
aid received from SRA, grants P1-0292-0101 and J1-9643-0101.} { }and 
G\'abor Luk\'acs\thanks{The second author gratefully acknowledges the 
generous financial support received from NSERC and the University of 
Manitoba, which enabled him to do this research.}}

\title{On zero-dimensionality and the connected component\\
of locally pseudocompact groups
}

\begin{document}

\makeatletter
\def\@fnsymbol#1{\ifcase#1\or * \or 1 \or 2  \else\@ctrerr\fi\relax}

\let\mytitle\@title
\makeatother

\chead{{\fontfamily{ppl}\fontsize{14}{14}\selectfont\itshape
Dedicated to Wis Comfort on the occassion of his 77th 
birthday}}

\maketitle

\def\thanks#1{} 

\thispagestyle{fancy}



\begin{abstract} 
A topological group is {\em locally pseudocompact} if it contains a 
non-empty open set with pseudocompact closure. In this note, we prove that 
if $G$ is a group with the property that every closed subgroup of $G$ is 
locally pseudocompact, then $G_0$ is dense in the 
component of the completion of $G\hspace{-1pt}$, and 
$G/G_0$ is zero-dimensional. We also provide examples of hereditarily 
disconnected pseudocompact groups with strong minimality properties of 
arbitrarily large dimension, and thus show that $G/G_0$ may fail to be 
zero-dimensional even for totally minimal pseudocompact groups.
\end{abstract}

\section{Introduction}

A Tychonoff space is {\itshape zero-dimensional} if it has a base 
consisting of {\itshape clopen} (open-and-closed)  sets.
With each topological group $G$ are associated functorial subgroups 
related to connectedness properties of $G\hspace{-1.5pt}$, defined as 
follows (cf.~\cite[1.1.1]{Dikconcomp}):
\begin{myalphlist}

\item
$G_0$ denotes the connected component of the identity;

\item
$q(G)$ denotes the {\em quasi-component} of the identity, that is, the 
intersection of all clopen sets containing the identity;

\item
$z(G)$ denotes the 
intersection of all kernels of continuous homomorphisms 
from $G$ into zero-dimensional groups;

\item
\mbox{$o(G)$} denotes the 
intersection of all open subgroups of $G\hspace{-1.5pt}$.

\vskip 1pt

\end{myalphlist}
It is well known  that these subgroups are closed and normal
(cf.~\cite[7.1]{HewRos}, \cite[2.2]{DikCOTA}, and~\cite[1.32(b)]{GLCLTG}).
Clearly, \mbox{$G_0\hspace{-2pt}\subseteq \hspace{-2pt} q(G)
\hspace{-2pt} \subseteq \hspace{-2pt} z(G)
\hspace{-2pt} \subseteq \hspace{-2pt}o(G)$}, and all four are equal 
for locally compact groups.

\begin{ftheorem}[\pcite{7.7, 7.8}{HewRos}] \label{intro:thm:LC}
Let $L$ be a locally compact group. Then 
$L/L_0$ is zero-dimensional and
\mbox{$L_0\hspace{-2pt} = \hspace{-2pt} q(L) 
\hspace{-2pt} = \hspace{-2pt} z(L)
\hspace{-2pt} = \hspace{-2pt}o(L)$}.
\end{ftheorem}

The aim of the present paper is to investigate to what extent the 
condition of local compactness can be relaxed in Theorem~\ref{intro:thm:LC}.
Although Theorem~\ref{intro:thm:LC} might appear as a result about 
connectedness,  it has far more to do with different 
degrees of {\itshape disconnectedness}. Recall that
a space $X$ is {\itshape hereditarily disconnected} if its 
connected components are singletons, and $X$ is {\itshape totally 
disconnected} if its quasi-components are singletons. Clearly, 
\begin{align*}
\mbox{zero-dimensional} 
\stackrel{(*)}{\Longrightarrow}
\mbox{totally disconnected}
\stackrel{(**)}{\Longrightarrow}  
\mbox{hereditarily disconnected},
\end{align*}
and by Vedenissoff's classic theorem, both implications are 
reversible for locally compact (Hausdorff) spaces, that is, the three 
properties are equivalent for such spaces (cf.~\cite{Vedenissoff}).

\chead{\small\itshape D. Dikranjan and G. Luk\'acs  / Locally pseudocompact groups}
\fancyhead[RO,LE]{\small \thepage}

It is well known that the quotient $G/G_0$ is hereditarily disconnected 
for every topological group~$G$ (cf.~\cite[7.3]{HewRos} 
and~\cite[1.32(c)]{GLCLTG}). Thus, if the implications ($*$) and ($**$) 
are reversible for $G/G_0$, then $G/G_0$ is zero-dimensional, and so
\mbox{$G_0\hspace{-2pt} = \hspace{-2pt} q(G)
\hspace{-2pt} = \hspace{-2pt} z(G)$}. If in addition
\mbox{$z(G)\hspace{-2pt} = \hspace{-2pt} o(G)$}, then
Theorem~\ref{intro:thm:LC} holds for $G\hspace{-1pt}$. This phenomenon 
warrants introducing some terminology.

\begin{definition}
A topological group $G$ is {\itshape Vedenissoff} if the quotient
$G/G_0$ is zero-dimensional; if in addition 
\mbox{$z(G)\hspace{-2pt} = \hspace{-2pt} o(G)$}, then we say that $G$ is
{\itshape strongly Vedenissoff}.
\end{definition}

Our goal is to identify classes of (strongly) Vedenissoff groups, and to 
find examples of non-Vedenissoff groups that have many compactness-like 
properties. The latter will demonstrate how close a group must be to 
being locally compact (or compact) in order to be Vedenissoff. (Not every 
Vedenissoff group is strongly Vedenissoff. Indeed, $\mathbb{Q}/\mathbb{Z}$ 
is zero-dimensional, but has no proper open subgroups, and so
\mbox{$z(\mathbb{Q}/\mathbb{Z})\hspace{-2pt} \neq
\hspace{-2pt} o(\mathbb{Q}/\mathbb{Z})$}. However, thanks to 
Theorem~\ref{prel:thm:connsum}(a) below, these two notions coincide in 
the class of groups that are considered in this paper.)

\bigskip

A Tychonoff space $X$ is {\em pseudocompact} if every 
continuous real-valued map on $X$ is bounded. A~topological group $G$ is 
{\em locally pseudocompact} if there is a neighborhood $U$ of the identity 
such that $\operatorname{cl}_G U$ is pseudocompact. (Clearly, every 
metrizable locally pseudocompact group is locally compact.)
We say that $G$ is 
{\itshape hereditarily} [{\itshape locally}]
{\itshape pseudocompact} if every closed subgroup of $G$ is 
[locally] pseudocompact. (Note that the adjective {\itshape hereditary} 
applies only to closed subgroups here, and not to all subgroups. Indeed, 
by Corollary~\ref{prel:cor:subLPS} below, if every subgroup of 
a~topological group is locally pseudocompact, then  the group is 
discrete, which is of no interest for the present paper.)

More than fifteen years ago, Dikranjan showed that hereditarily 
pseudocompact groups are strongly Vedenissoff (cf.~\cite[1.2]{DikPS0dim}). 
We obtain in this paper a theorem that simultaneously generalizes 
Theorem~\ref{intro:thm:LC}, this result of D.D., and
provides a positive solution to a problem posed by 
Comfort and Luk\'acs (cf.~\cite[4.17]{ComfGL}).

\begin{Ltheorem} \label{thm:main:herLPS}
Let $G$ be a hereditarily locally pseudocompact group. Then 
$G\hspace{-0.5pt}/G_0$ is 
zero-dimensional and \mbox{$G_0\hspace{-2pt} = \hspace{-2pt} q(G)
\hspace{-2pt} = \hspace{-2pt} z(G)\hspace{-2pt} = \hspace{-2pt}o(G)$}, 
that is, $G$ is strongly Vedenissoff.
\end{Ltheorem}

The next example shows that the condition of hereditarily local 
pseudocompactness in Theorem~\ref{thm:main:herLPS} cannot be replaced with 
(local) pseudocompactness.

\begin{example} \label{intro:ex:CvM}
Comfort and van Mill showed that for  every natural number $n$ there 
exists an abelian pseudocompact group $G_n$ such 
that $G_n$ is totally disconnected, but 
\mbox{$\dim G_n \hspace{-2pt} = \hspace{-1pt} n$}
(cf.~\cite[7.7]{ComfvMill2}). In particular, the converse of the 
implication ($*$) may fail
for these groups~$G_n$, and they are not Vedenissoff. This shows that 
pseudocompact groups need not be Vedenissoff.
\end{example}

Although pseudocompactness alone is too weak  a property to imply that 
the group is Vedenissoff, it turns out that it is sufficient in the 
presence of some additional compactness-like properties. Recall that a 
(Hausdorff) topological group $G$ is {\itshape minimal} if there is no 
coarser (Hausdorff) group topology (cf.~\cite{Steph} and~\cite{Doi}), and 
$G$ is {\itshape totally minimal} if every (Hausdorff) quotient of $G$ is 
minimal (cf.~\cite{DikPro}). Equivalently, $G$ is totally minimal if every 
continuous surjective homomorphism \mbox{$G\rightarrow H$} is open.

An unpublished result of Shakhmatov states that the converse of ($*$) 
holds for minimal pseudocompact groups. Specifically, Shakhmatov proved 
that every pseudocompact  totally disconnected group admits a coarser 
zero-dimensional group topology, and thus minimal pseudocompact totally 
disconnected groups are zero-dimensional (cf.~\cite[1.6]{DikPS0dim}). We 
prove a generalization of Shakhmatov's result:

\begin{Ltheorem}\mbox{ }

\label{thm:main:coarser}

\begin{myalphlist}

\item
Every locally pseudocompact totally disconnected group admits a coarser
zero-dimensional group topology.

\item
Every minimal, locally pseudocompact, totally disconnected group is 
zero-dimensional, and thus strongly Vedenissoff.

\end{myalphlist}
\end{Ltheorem}

\begin{Ltheorem} \label{thm:main:totminLPS}
Let $G$ be a totally minimal locally pseudocompact group. Then 
\mbox{$G_0 \hspace{-2pt} = \hspace{-2pt} q(G)$} if and only if 
$G/G_0$ is   zero-dimensional, in which case
$G$ is strongly Vedenissoff.
\end{Ltheorem}

More than twenty years ago,  Arhangel$'$ski\u{\i} asked whether every 
totally disconnected topological group admits a coarser zero-dimensional 
group topology. Megrelishvilli answered this question in the negative by 
constructing a minimal totally disconnected group that is not 
zero-dimensional (cf.~\cite{Megrel4}). In particular, the converse of the 
implication ($*$) fails for minimal groups. Megrelishvilli's example
also shows that local pseudocompactness cannot be omitted from 
Theorem~\ref{thm:main:coarser}.

Our last result is a negative one, and it is a far reaching extension of 
the result of Comfort and van Mill cited in Example~\ref{intro:ex:CvM}.
Recall that a 
group $G$ is {\itshape perfectly} ({\itshape totally}) {\itshape minimal} 
if the product \mbox{$G \hspace{-1.6pt}\times\hspace{-2.5pt} H$} 
is ({\itshape totally}) {\itshape 
minimal} for every ({\itshape totally}) {\itshape minimal} group 
$H\hspace{-1.5pt}$ (cf.~\cite{Stoy}).

\begin{Ltheorem} \label{thm:main:example}
For every natural number $n$ or \mbox{$n\hspace{-2pt} =\hspace{-1pt} 
\omega$}, there exists an abelian pseudocompact group $G_n$ such that 
$G_n$ is perfectly totally minimal, hereditarily disconnected, but
\mbox{$\dim G_n \hspace{-2pt} = \hspace{-1pt}n$}.
\end{Ltheorem}

There are many known examples of pseudocompact groups for which the 
equality \mbox{$G_0\hspace{-2pt}= \hspace{-2pt} q(G)$} fails 
(cf.~\cite[Theorem~11]{Dikdimpsc}, \cite[1.4.10]{Dikconcomp}, 
and~\cite[4.7, 5.5]{ComfGL}). By Theorem~\ref{thm:main:totminLPS},
one has \mbox{$(G_n)_0\hspace{-2pt} \neq \hspace{-2pt} q(G_n)$}~for each 
of the groups $G_n$ provided by Theorem~\ref{thm:main:example}, and thus
the $G_n$ are not totally disconnected. This shows that the converse of 
the implication ($**$) may fail for totally minimal pseudocompact groups.

\bigskip

The paper is structured as follows. In \S\ref{sect:prel}, we recall some 
well-known facts on locally pseudocompact and locally compact groups, 
their $G_\delta$-topologies, and their connectedness properties.
We devote \S\ref{sect:herLPS} to the proof of 
Theorem~\ref{thm:main:herLPS}, while the proofs of 
Theorems~\ref{thm:main:coarser} and~\ref{thm:main:totminLPS} are presented 
in \S\ref{sect:short}. Finally, in \S\ref{sect:exmp}, we prove a general 
theorem concerning embedding of groups with minimality properties as 
quasi-components of pseudocompact groups with the same minimality 
properties, which yields Theorem~\ref{thm:main:example}.

\section{Preliminaries}

\label{sect:prel}

All topological groups here are assumed to be Hausdorff, and thus 
Tychonoff (cf.~\cite[8.4]{HewRos} and \cite[1.21]{GLCLTG}). Except when 
specifically noted, no algebraic assumptions are imposed on the groups; in 
particular, our groups are not necessarily abelian. A ``neighborhood" of a 
point means an {\em open} set containing the point.

Although, in general, there are a number of useful uniform structures on a 
topological group that induce its topology, in this note, we adhere to the 
two-sided uniformity and the notions of precompactness and completeness 
that derive from it (cf. \cite{RoeDie}, \cite{We3}, \cite{Rai}, 
\cite[(4.11)-(4.15)]{HewRos}, and \cite[Section 1.3]{GLCLTG}). A 
fundamental property of this notion of completeness is that for every 
topological group $G\hspace{-1pt}$, there is a complete topological group 
$\widetilde G$ (unique up to a topological isomorphism) that contains $G$ 
as a dense topological subgroup; in other words, $\widetilde G$ is a {\em 
group completion} of $G$ (cf. \cite{Rai} and \cite[1.46]{GLCLTG}).

\begin{ftheorem}[\pcite{1.49(a), 1.51}{GLCLTG}] \mbox{ }
\label{prel:thm:comp}

\begin{myalphlist}

\item
Let $G$ be a topological group, and $H$ a subgroup. Then
\mbox{$\widetilde H  \hspace{-2.5pt}=
\hspace{-1pt} \operatorname{cl}_{\widetilde G} H\hspace{-1.5pt}$.}

\item
If $G$ is a locally compact group, then $G$ is complete, that is,
\mbox{$\widetilde G\hspace{-2.25pt}= \hspace{-2pt}G\hspace{-1pt}$.}

\end{myalphlist}
\end{ftheorem}

A subset $X$ of a topological group  $G$ is {\em precompact} if for every
neighborhood $U$ of the identity,
there is a finite \mbox{$S \hspace{-2pt} \subseteq \hspace{-2pt}X$} such 
that  \mbox{$X \hspace{-2.5pt}\subseteq \hspace{-2pt} (SU) 
\hspace{-2.5pt}\cap \hspace{-2.5pt} (U \hspace{-1pt}S)\hspace{-1pt}$}.
(Some authors refer to precompact sets as {\em bounded} ones.)
A~topological group $G$ is {\em locally precompact} if 
$G$ admits a base of precompact neighborhoods at the identity. 
Since every pseudocompact subset of a topological group is precompact 
(cf.~\cite[1.11]{ComfTrig}), locally pseudocompact groups are locally 
precompact.

Weil showed in 1937 that 
the completion of a locally precompact group with respect to its left or 
right uniformity admits the structure of a
locally compact {\em group} containing $G$ as a dense topological subgroup 
(cf. \cite{We3}). This (one-sided) {\it Weil-completion} coincides
with the Ra\v{\i}kov-completion $\widetilde G$ constructed in 1946 
(cf.~\cite{Rai}). Therefore,   $G$ is 
locally precompact if and only if $\widetilde G$ is locally compact.

Theorem~\ref{prel:thm:lcps} below, which summarizes the main results 
of~\cite{ComfRoss2} and~\cite{ComfTrig}, provides a  characterization of 
(locally) pseudocompact groups. Recall that
a {\em $G_\delta$-subset}~\hspace{0.5pt}of a space 
$(X,\mathcal{T})$ is 
a set of the form 
\mbox{$\bigcap\limits_{n < \omega} \hspace{-1.5pt} U_n$} with 
each \mbox{$U_n \hspace{-2pt}\in \hspace{-2pt} 
\mathcal{T}\hspace{-1.5pt}$}.
The {\em $G_\delta$-topology}  
on~$X$ is the topology generated by the $G_\delta$-subsets of 
$(X,\mathcal{T})$. A subset of $X$ is 
{\em  $G_\delta$-open} (respectively, {\em $G_\delta$-closed}, {\em 
$G_\delta$-dense}) if it is open (respectively, closed, dense) in 
the $G_\delta$-topology on $X\hspace{-1pt}$.

\begin{ftheorem}[\cite{ComfRoss2} and \cite{ComfTrig}] 
\label{prel:thm:lcps}
A topological group $G$ is [locally] pseudocompact if and 
only if $G$ is [locally] precompact and $G_\delta$-dense in $\widetilde G$,
in which case 
\mbox{$\widetilde G \hspace{-1.5pt} = \hspace{-2pt}\beta G$}
[\mbox{$\beta \widetilde G \hspace{-1.5pt} = \hspace{-2pt}\beta G$}].
\end{ftheorem}

Since the $G_\delta$-topology of groups plays an important 
role in the present work, we introduce some notations, and then record a 
few useful facts. We let $\Lambda(G)$ denote the set of closed 
$G_\delta$\mbox{-}subgroups of the topological group $G\hspace{-1pt}$, 
that is, closed subgroups of $G$ that are also $G_\delta$-subsets of 
$G\hspace{-1pt}$, and we set
\mbox{$\Lambda_c (G) \hspace{-2pt} : = \hspace{-2pt}
\{K \hspace{-3pt}\in\hspace{-2pt} \Lambda(G)
\mid K \mbox{ is compact} \}$} and
\mbox{$\Lambda_c^* (G) \hspace{-2pt} : = \hspace{-2pt}
\{K \hspace{-3pt}\in\hspace{-2pt} \Lambda_c(G)  \mid 
K\hspace{-2pt} \triangleleft\hspace{-2pt} G\}$}.

\begin{ftheorem} \label{prel:thm:delta}
Let $G$ be a topological group. Then:

\begin{myalphlist}

\item
{\rm (\cite[2.5]{GLCLTG})} the $G_\delta$-topology is a group topology on 
$G$;

\item
{\rm (\cite[8.7]{HewRos})} if $G$ is locally compact, then $\Lambda_c(G)$ 
is a base at the identity for the $G_\delta$-topology on~$G$;

\item
{\rm (\cite[8.7]{HewRos})} if $G$ is locally compact and $\sigma$-compact, 
then $\Lambda_c^*(G)$ is a base at the 
identity for the $G_\delta$-topology on $G\hspace{-1pt}$.

\end{myalphlist}
\end{ftheorem}

We have already mentioned that the adjective {\itshape hereditary} used in 
the term {\itshape hereditarily locally pseudocompact} applies only to the 
closed subgroups of a given group. The next theorem shows that it would be 
uninteresting to interpret {\itshape hereditary} as applying to all 
subgroups.

\begin{theorem} \label{prel:thm:subLPS}
Let $G$ be a locally pseudocompact group. If every countable subgroup of 
$G$ is locally pseudocompact, then $G$ is discrete.
\end{theorem}

Before we proceed to the proof of Theorem~\ref{prel:thm:subLPS},
we formulate a well-known observation that will be frequently used 
later on too.

\begin{lemma} \label{herLPS:lemma:dense}
Let $H$ be a topological group, $D$ a dense subgroup, and $O$ an open 
subgroup of $H\hspace{-2pt}$. 
Then \mbox{$\operatorname{cl}_H(D\hspace{-2pt}\cap\hspace{-2pt} O) 
\hspace{-2pt} =  \hspace{-1.75pt}O\hspace{-0.5pt}$}. 
\end{lemma}

\begin{proof}
Since $D$ is dense and $O$ is open in $H\hspace{-1.75pt}$, one has
\mbox{$\operatorname{cl}_H O \hspace{-2pt}= \hspace{-2pt} 
\operatorname{cl}_H (D\hspace{-2pt}\cap\hspace{-2pt}O)$}. On the other 
hand, $O$ is closed in $H\hspace{-1.75pt}$,  because every open subgroup 
of a topological group is also closed (cf.~\cite[5.5]{HewRos} 
and~\cite[1.10(c)]{GLCLTG}). This completes the proof.
\end{proof}

\begin{proof}[Proof of Theorem~\ref{prel:thm:subLPS}.] The proof consists 
of two steps.

{\itshape Step 1.} We show that $G$ is locally compact and metrizable.
Since $G$ is locally precompact, its completion \mbox{$L \hspace{-2pt} := 
\hspace{-2pt} \widetilde G$} is locally compact, and thus by 
Theorem~\ref{prel:thm:delta}(b), $\Lambda_c(L)$ is a base at the identity 
of the $G_\delta$-topology on $L$. Pick 
\mbox{$K \hspace{-2pt} \in \hspace{-2pt} \Lambda_c(L)$}.
Since $G$ is locally pseudocompact, by Theorem~\ref{prel:thm:lcps},
$G$~is $G_\delta$-dense in $L$.
By Theorem~\ref{prel:thm:delta}(a), the $G_\delta$-topology 
is  a group topology on~$L$,~and so  Lemma~\ref{herLPS:lemma:dense}~is 
applicable. Consequently, by Lemma~\ref{herLPS:lemma:dense},
\mbox{$P \hspace{-2pt}:= \hspace{-2pt} K \hspace{-2pt}\cap\hspace{-2pt} G$} 
is $G_\delta$-dense in $K\hspace{-1.5pt}$.

We claim that $P$ is finite. Let $S$ be a countable subgroup of $P$. Then, 
by our assumption, $S$~is locally pseudocompact, being a countable 
subgroup of $G \hspace{-1pt}$. However, $S$ is also a subgroup of the 
compact group $K$, and thus, by Theorem~\ref{prel:thm:comp}, its 
completion $\widetilde S$ is compact. Consequently, by 
Theorem~\ref{prel:thm:lcps}, $S$ is pseudocompact. Hence, $S$ is finite, 
because there are no countably infinite homogeneous pseudocompact spaces
(cf.~\cite[1.3]{vanDouwPS}). This shows that $P$ is finite. 

Since $P$ is $G_\delta$\mbox{-}dense in $K\hspace{-1pt}$, it follows that 
$K$ is finite, and so $L$ has countable pseudocharacter. This implies that 
$L$ is metrizable, because every locally compact space of countable 
pseudocharacter is first countable (cf.~\cite[3.3.4]{Engel6}). Therefore, 
\mbox{$L\hspace{-2pt} = \hspace{-2pt} G$}, because by $G$ is 
$G_\delta$-dense in $L$.

{\itshape Step 2.} We show that $G$ is discrete. Since $G$ is locally 
compact, there is a neighborhood $U$ of the identity in $G$ such that 
$\operatorname{cl}_G U$ is compact. Then the open subgroup
\mbox{$G^\prime \hspace{-2pt} := \hspace{-2pt} \langle U \rangle$} 
generated by $U$ is compactly generated, and in particular, 
$\sigma$-compact (cf.~\cite[5.12, 5.13]{HewRos}). Thus, $G^\prime$ is 
separable, because it is metrizable, being a subgroup of $G\hspace{-1pt}$.
Let $S$ be a countable dense subgroup of $G^\prime$. By our assumption, 
$S$ is locally pseudocompact, and thus it is locally compact, because $G$ 
is metrizable. So, $S$ is a countable locally compact group, and therefore
$S$ is discrete. In particular, $S$ is complete, and 
\mbox{$G^\prime\hspace{-2pt}=\hspace{-2pt} S\hspace{-1pt}$} is discrete. 
Hence, $G$ is also discrete. 
\end{proof}

\begin{corollary} \label{prel:cor:subLPS}
If every subgroup of a topological group $G$ is locally pseudocompact, 
then $G$ is discrete. \qed
\end{corollary}

Finally, we summarize the relationship between connectedness properties of 
locally pseudocompact groups and their completions.

\begin{ftheorem} \label{prel:thm:connsum}
Let $G$ be a locally pseudocompact group. Then:

\begin{myalphlist}

\item
{\rm (\cite[1.4]{Dikconpsc}, \cite[4.9]{ComfGL})}
\mbox{$q(G)\hspace{-2.1pt}=\hspace{-1.9pt} o(G) 
\hspace{-2.1pt}=\hspace{-1.9pt} 
(\widetilde G)_0 \hspace{-2.1pt} \cap\hspace{-2pt} G$};

\item
{\rm (\cite{TkaDimLP}, \cite[4.12]{ComfGL})}
$G$ is zero-dimensional if and only if $\widetilde G$ is zero-dimensional;

\item
{\rm (\cite[1.7]{Dikconpsc}, \cite[4.15(b)]{ComfGL})}
$G/G_0$ is zero-dimensional if and only if $G_0$ is dense in 
$(\widetilde G)_0$, in which case 
\mbox{$G_0 \hspace{-2pt} = \hspace{-1.5pt} q(G)$}.

\end{myalphlist}
\end{ftheorem}

\section{Proof of Theorem~\ref{thm:main:herLPS}}

\label{sect:herLPS}

\begin{Ltheorem*}[\ref{thm:main:herLPS}]
Let $G$ be a hereditarily locally pseudocompact group. Then $G/G_0$ is 
zero-dimensional and \mbox{$G_0\hspace{-2pt} = \hspace{-2pt} q(G)
\hspace{-2pt} = \hspace{-2pt} z(G)\hspace{-2pt} = \hspace{-2pt}o(G)$}, 
that is, $G$ is strongly Vedenissoff.
\end{Ltheorem*}

In this section, we present the proof of Theorem~\ref{thm:main:herLPS}. 
By Theorem~\ref{prel:thm:connsum}(a),
\mbox{$q(G)\hspace{-2pt} = \hspace{-2pt} z(G)\hspace{-2pt} = 
\hspace{-2pt}o(G)$} for every locally pseudocompact  group 
$G\hspace{-1pt}$. We have already noted that~if~$G/G_0$  is 
zero-dimensional, then 
\mbox{$G_0\hspace{-2pt} = \hspace{-2pt} q(G)
\hspace{-2pt} = \hspace{-2pt} z(G)$}. Thus, it suffices to show that 
$G/G_0$ is zero-dimensional for every hereditarily locally 
pseudocompact group $G\hspace{-1pt}$. 
Since every (Hausdorff) quotient of a hereditarily locally pseudocompact 
group is again hereditarily locally pseudocompact, and the quotient 
$G/G_0$ is hereditarily disconnected  (cf.~\cite[7.3]{HewRos} 
and~\cite[1.32(c)]{GLCLTG}), it suffices to prove the following statement.

\begin{theorem} \label{herLPS:thm:G0}
Let $G$ be a hereditarily locally pseudocompact, hereditarily disconnected
group. Then $\widetilde G$ is hereditarily disconnected, and $G$ is 
zero-dimensional.
\end{theorem}

In the setting of Theorem~\ref{herLPS:thm:G0}, if $\widetilde G$ is 
hereditarily disconnected, then by Theorem~\ref{intro:thm:LC}, $\widetilde 
G$ is zero-dimensional (because it is locally compact), and so $G$ is 
zero-dimensional too. Thus, it suffices to show that $\widetilde G$ is 
hereditarily disconnected whenever $G$ is so. We prove the contrapositive 
of this statement, namely, that if $(\widetilde G)_0$ is non-trivial, then 
$G_0$ is 
non-trivial too. The proof is broken down into several steps: 
First, it is shown in Proposition~\ref{herLPS:prop:NxR} that  
Theorem~\ref{herLPS:thm:G0} holds in the case where the completion 
$\widetilde G$ of $G$ is a direct product of a zero-dimensional compact 
group and the real line $\mathbb{R}$. Then, in 
Proposition~\ref{herLPS:prop:compact}, it 
is proven that if $\widetilde G$ contains a non-trivial compact connected 
subgroup, 
then $G_0$ is non-trivial.  Finally, it is shown that if the component of 
$\widetilde G$ is non-trivial, but contains no compact connected subgroup, 
then $G_0$ contains $\mathbb{R}$ as a closed subgroup.

\begin{proposition} \label{herLPS:prop:NxR}
Let $N$ be a zero-dimensional compact group, and $G$ a dense 
hereditarily locally pseudocompact
subgroup of \mbox{$N \hspace{-2.5pt}\times \hspace{-1.6pt} \mathbb{R}$}.
Then one has 
\mbox{$\{e\}\hspace{-2.5pt}\times \hspace{-1.6pt} \mathbb{R}
\hspace{-2pt}\subseteq \hspace{-2pt}G\hspace{-1pt}$}.
\end{proposition}

In order to prove Proposition~\ref{herLPS:prop:NxR}, we recall a notion 
and a result that is well-known to profinite group theorists. A 
topological group $P$ is {\itshape topologically finitely generated} if it 
contains a dense finitely generated group, that is, there exists a finite 
subset $F$ of $P$ such that \mbox{$P\hspace{-2pt} = \hspace{-2pt} 
\operatorname{cl}_P \langle F \rangle$}.

\begin{ftheorem}[{\cite[2.5.1]{RibesZales}, \cite[2.1(a)]{DikTkaTkaTHIN}}] 
\label{herLPS:thm:fg}
If $P$ is a topologically finitely generated compact zero-dimensional 
group, then $P$ is metrizable.
\end{ftheorem}

\begin{proof}[Proof of Proposition~\ref{herLPS:prop:NxR}.]
As $G$ is locally pseudocompact, by Theorem~\ref{prel:thm:lcps},
$G$ is $G_\delta\mbox{-}$dense in 
\mbox{$N \hspace{-2.5pt}\times \hspace{-1.6pt} \mathbb{R}$}.
The set \mbox{$A_r \hspace{-2pt} := \hspace{-2pt} 
N \hspace{-2.5pt}\times \hspace{-1.6pt}\{r\}$} is a $G_\delta$-set 
in \mbox{$N \hspace{-2.5pt}\times \hspace{-1.6pt} \mathbb{R}$} for every
\mbox{$r\hspace{-2pt} \in \hspace{-2pt} \mathbb{R}$}, and so
\mbox{$A_r\hspace{-2.5pt} \cap \hspace{-2pt} G 
\hspace{-2pt} \neq \hspace{-2pt} \emptyset$}. Thus,~for every
\mbox{$r\hspace{-2pt} \in \hspace{-2pt} \mathbb{R}$}, there is 
\mbox{$g\hspace{-2pt} \in \hspace{-2pt} N$} such that
\mbox{$(g,r) \hspace{-2pt}\in\hspace{-2pt} G$}.

Let \mbox{$g_1,g_2\hspace{-2pt} \in \hspace{-2pt} N$}  be such that
\mbox{$x_1:=(g_1,1) \hspace{-2pt}\in\hspace{-2pt} G$} and
\mbox{$x_2:=(g_2,\sqrt{2}) \hspace{-2pt}\in\hspace{-2pt} G$}, and put
\mbox{$P\hspace{-2pt}:= \hspace{-1pt} \operatorname{cl}_N\langle 
g_1,g_2\rangle$} and 
\mbox{$H\hspace{-2pt}:= \hspace{-1pt} \operatorname{cl}_G\langle
x_1,x_2\rangle$}. Since $P$ is a closed subgroup of $N\hspace{-2pt}$, it 
is a compact topologically finitely generated zero-dimensional group, and 
by Theorem~\ref{herLPS:thm:fg}, $P$ is metrizable. Thus, the product
\mbox{$P \hspace{-2.5pt}\times \hspace{-1.6pt} 
\mathbb{R}\hspace{0.5pt}$}~is metrizable, and so 
$\hspace{-0.25pt}H\hspace{-0.25pt}$ is metrizable, being
a subgroup of \mbox{$P \hspace{-2.75pt}\times \hspace{-2.1pt}
\mathbb{R}$}. On the other hand, $H$ is 
\mbox{locally}~pseudocompact, being 
a closed subgroup of the hereditarily locally 
pseudocompact group~$G\hspace{-1pt}$. Therefore, $H$ is locally 
compact. Hence, by Theorem~\ref{prel:thm:comp},
$H$~is closed not only in $G\hspace{-1pt}$, but also in
\mbox{$N \hspace{-2.5pt}\times \hspace{-1.6pt}\mathbb{R}$}.

Let \mbox{$\pi_2 \colon N \hspace{-2.5pt}\times \hspace{-1.6pt}\mathbb{R}
\rightarrow  \mathbb{R}$} denote the second projection, and put
\mbox{$\pi \hspace{-2pt} := \hspace{-2pt} \pi_{2|H}$}. Since $N$ is
compact and $H$ is closed in \mbox{$N \hspace{-2.5pt}\times 
\hspace{-1.6pt}\mathbb{R}$}, 
$\pi_2$ is a closed map (cf.~\cite[3.1.16]{Engel6}), and thus 
$\pi(H)$ is closed in $\mathbb{R}$ and $\pi$ is a closed map too.  
This implies that  $\pi$ is surjective, 
because $\pi(H)$ contains the dense subgroup $\langle 1,\sqrt{2}\rangle$ 
of $\mathbb{R}$. Consequently, $\pi$ is a quotient map, and $\mathbb{R}$ 
is topologically isomorphic to a quotient of $H\hspace{-1.5pt}$.
Therefore, $H$ is not zero-dimensional, and by Theorem~\ref{intro:thm:LC},
$H_0$ is non-trivial.

Since \mbox{$H_0\hspace{-2pt} \subseteq \hspace{-2pt} 
\{e\}\hspace{-2.5pt}\times \hspace{-1.6pt} \mathbb{R} = 
(N\hspace{-2.5pt}\times \hspace{-1.6pt} \mathbb{R})_0$}, one has
\mbox{$H_0\hspace{-2pt} = \hspace{-2pt}
\{e\}\hspace{-2.5pt}\times \hspace{-1.6pt} \mathbb{R}$}, because 
$\mathbb{R}$ has no non-trivial proper connected subgroups.
Hence, \mbox{$\{e\}\hspace{-2.5pt}\times \hspace{-1.6pt} \mathbb{R}
\hspace{-2pt} = \hspace{-2pt} H_0 \hspace{-2pt}\subseteq\hspace{-2pt} H 
\hspace{-2pt}\subseteq\hspace{-2pt} G$}, as desired.
\end{proof}

\begin{proposition} \label{herLPS:prop:compact}
Let $G$ be a hereditarily locally pseudocompact group such that 
the completion $\widetilde G$ is $\sigma$-compact. If $\widetilde G$
contains a non-trivial compact connected subgroup, then $G_0$ is 
non-trivial.
\end{proposition}

As one may expect, the proof of Proposition~\ref{herLPS:prop:compact} 
relies on Dikranjan's result for hereditarily pseudocompact groups.

\begin{ftheorem}[\pcite{1.2, 2.6}{DikPS0dim}] \label{herLPS:thm:herpscp}
Let $G$ be a hereditarily pseudocompact group. Then $G/G_0$ is 
zero-dimensional, \mbox{$G_0 \hspace{-2pt} = \hspace{-2pt} q(G)$}, and
$G_0$ is dense in $(\widetilde G)_0$, that is, $G$ is strongly 
Vedenissoff.
\end{ftheorem}

\begin{proof}[Proof of Proposition~\ref{herLPS:prop:compact}.]
Let $C$ be a non-trivial connected compact subgroup of 
$\widetilde{G}\hspace{-1pt}$.~By 
Theorem~\ref{prel:thm:delta}(c), 
$\Lambda_c^*(\widetilde G)$ is a~base 
at the identity for the $G_\delta$-topology on 
$\widetilde G\hspace{-1pt}$. So, we may pick
\mbox{$K \hspace{-3pt} \in \hspace{-2pt} \Lambda_c^*(\widetilde G)$}.
Since $K$ is a~normal subgroup of $\widetilde G\hspace{-1pt}$, the set
$KC$ is a subgroup of $\widetilde G\hspace{-1pt}$, and 
$KC$ is compact, because both $K$ and $C$ are compact. Furthermore,
$KC$ is $G_\delta$-open in $\widetilde G\hspace{-1pt}$, as it contains the 
$G_\delta$-set $K\hspace{-1.5pt}$. By Theorem~\ref{prel:thm:lcps},
$G$~is $G_\delta$-dense in $\widetilde G\hspace{-1pt}$. 
By Theorem~\ref{prel:thm:delta}(a), the $G_\delta$-topology 
is  a group topology on~$\widetilde G\hspace{-1pt}$,~and so 
Lemma~\ref{herLPS:lemma:dense} is applicable. 
Consequently, by 
Lemma~\ref{herLPS:lemma:dense},
\mbox{$P \hspace{-2pt}:= \hspace{-2pt} KC \hspace{-2pt}\cap\hspace{-2pt} 
G$} is $G_\delta$-dense in $KC\hspace{-0.5pt}$; in particular, $P$ is 
dense in $KC\hspace{-1pt}$, and thus, by Theorem~\ref{prel:thm:comp},
\mbox{$\widetilde P \hspace{-2pt} = \hspace{-2pt} KC\hspace{-1.5pt}$}.

We show that $P$ is hereditarily pseudocompact.
 To that end, let $S$ be a closed subgroup 
of~$P\hspace{-2pt}$. Since $KC$ is compact, it is closed in 
$\widetilde G\hspace{-1pt}$, and so $P$ is a closed subgroup of 
$G\hspace{-1pt}$. Thus,~$S$~is~a~closed 
subgroup of $G$, and by our assumption, $S$ is locally pseudocompact. 
By Theorem~\ref{prel:thm:lcps}, this implies that $S$ is $G_\delta$-dense 
in its completion $\widetilde S\hspace{-1pt}$. The group
$\widetilde S$ is compact, being a closed 
subgroup~of~$K\hspace{-0.5pt}C\hspace{-1pt}$,~and 
therefore $S$ is not only locally pseudocompact, but also
pseudocompact. This shows that $P$ is hereditarily 
pseudocompact. Therefore, by Theorem~\ref{herLPS:thm:herpscp}, $P_0$ is 
dense in 
\mbox{$(\widetilde P)_0 \hspace{-2pt} = \hspace{-2pt} (KC)_0$}, and hence
\begin{align*}
\{e\} \neq
C \subseteq (KC)_0 \subseteq \operatorname{cl}_{KC} P_0 \subseteq
\operatorname{cl}_{\widetilde G} G_0.
\end{align*}
In particular, $G_0$ cannot be trivial, as desired.
\end{proof}

One last ingredient of the proof of Theorem~\ref{herLPS:thm:G0} is a 
result that is often referred to as Iwasawa's Theorem (although it also 
relies on the work of Yamabe).

\begin{ftheorem}[\pcite{Theorem 5'}{Yamabe1}, \pcite{Theorem~13}{Iwas}]
\label{conn:thm:Iwasawa}
Let $L$ be a connected locally compact group. Then there is a compact 
connected subgroup $C\hspace{-1pt}$, and closed subgroups 
\mbox{$H_1,\hspace{-1pt}\ldots,\hspace{-1.0pt}H_r$} such that 
each $H_i$ is topologically isomorphic to the additive group $\mathbb{R}$, 
and $L$ is homeomorphic to 
\mbox{$H_1 \hspace{-2pt} \times 
\hspace{-2pt}\cdots\hspace{-1pt} \times \hspace{-1pt} H_r \hspace{-2pt} 
\times \hspace{-2pt} C \hspace{-1pt}$}.
\end{ftheorem}

\begin{proof}[Proof of Theorem~\ref{herLPS:thm:G0}.] We prove the 
contrapositive of the theorem. Let $G$ be a
hereditarily locally pseudocompact group such that
\mbox{$(\widetilde G)_0 \hspace{-2pt} \neq \hspace{-2pt} \{e\}$}.
We show that \mbox{$G_0 \hspace{-2pt} \neq \hspace{-2pt} \{e\}$}.

{\itshape Step 1.}
As $G$ is locally precompact, its completion \mbox{$L \hspace{-2pt} := 
\hspace{-2pt} \widetilde G$} is locally compact.
Let $U$ be a~neighborhood of the identity in $L$ such that 
$\operatorname{cl}_L U$ is \mbox{compact}. Put 
\mbox{$L^\prime \hspace{-2.5pt} := \hspace{-2pt} \langle U \rangle$}, 
the subgroup generated by~$U\hspace{-2pt}$, and \mbox{$G^\prime 
\hspace{-2.5pt} :=\hspace{-2pt} L^\prime \hspace{-2pt} \cap \hspace{-2pt} G$}. 
We claim that by replacing  $G$ with $G^\prime$ if necessary, we may 
assume that $\widetilde G$ is $\sigma$-compact from the outset.

Since $L^\prime$ contains $U\hspace{-2pt}$, it is an open subgroup of 
$L$, and thus it is also closed (cf.~\cite[5.5]{HewRos} 
and~\cite[1.10(c)]{GLCLTG}). Consequently,  $G^\prime$ is open and closed 
in $G\hspace{-1pt}$, and $G^\prime$ is also hereditarily locally 
pseudocompact.  By Lemma~\ref{herLPS:lemma:dense}, one has
\mbox{$L^\prime \hspace{-2pt} = \hspace{-1pt}
\operatorname{cl}_L G^\prime\hspace{-2pt}$}, and so, by 
Theorem~\ref{prel:thm:comp},
\mbox{$L^\prime \hspace{-2pt} = \hspace{-2pt}
\widetilde{G^\prime}$\hspace{-1.5pt}}. As $L^\prime$ 
is generated by~$U\hspace{-2pt}$, it is compactly generated, and in 
particular, it is $\sigma$-compact (cf.~\cite[5.12, 5.13]{HewRos}).
Since $L^\prime$ is an open subgroup of~$L$, by
Theorem~\ref{intro:thm:LC},
\mbox{$L_0 
\hspace{-2pt}= \hspace{-2pt}o(L) \hspace{-2pt} \subseteq 
\hspace{-2pt} L^\prime \hspace{-2pt}$}, and therefore 
\mbox{$L_0 \hspace{-2pt}= \hspace{-2pt} L^\prime_0$} (because
\mbox{$L^\prime\hspace{-2pt} \subseteq \hspace{-2pt}L$} implies 
\mbox{$L^\prime_0\hspace{-2pt} \subseteq \hspace{-2pt}L_0$}). 
Similarly, one has \mbox{$G_0 \hspace{-2pt}= \hspace{-2pt} G^\prime_0$}.
Hence, it suffices to show that 
\mbox{$G^\prime_0 \hspace{-2pt}\neq \hspace{-2pt}\{e\}$}, and
by replacing $G$ with $G^\prime$ if necessary, we may assume 
that $\widetilde G$ is $\sigma$-compact.

{\itshape Step 2.} Since $\widetilde G$ is $\sigma$-compact, if 
$\widetilde G$ contains a non-trivial compact connected subgroup, then
by Proposition~\ref{herLPS:prop:compact}, $G_0$ is non-trivial, and we are 
done. Thus, from now on,  we  assume that $\widetilde G$ contains no 
non-trivial compact connected subgroups. 

If $N$ is a compact subgroup of 
$\widetilde G$, then \mbox{$N_0\hspace{-2pt} = \hspace{-2pt} \{e\}$}, and 
by Theorem~\ref{intro:thm:LC}, $N$ is zero-dimensional. Thus, every 
compact subgroup of $\widetilde G$ is zero-dimensional.
In particular, $(\widetilde G)_0$~contains 
no non-trivial compact connected subgroups. Therefore, by 
Theorem~\ref{conn:thm:Iwasawa}, our assumption 
\mbox{$(\widetilde G)_0 \hspace{-2pt} \neq \hspace{-2pt} \{e\}$} yields
that there is a closed subgroup $H$ of $\widetilde G$ such that 
\mbox{$H\hspace{-2pt} \cong \hspace{-2pt} \mathbb{R}$} (that is, 
$H$ is topologically isomorphic to the additive group $\mathbb{R}$).

{\itshape Step 3.} 
By Theorem~\ref{prel:thm:delta}(c), 
$\Lambda_c^*(\widetilde G)$~is a~base 
at the identity for the $G_\delta$-topology on 
$\widetilde G\hspace{-1pt}$,~and so we may pick
\mbox{$N \hspace{-3pt} \in \hspace{-2pt} \Lambda_c^*(\widetilde G)$}.
Since $N$ is a 
normal subgroup, $(\widetilde G)_0$ acts continuously on $N$ by 
conjugation, and the orbit of $x$ is a connected subspace of 
$N\hspace{-2pt}$. By Step 2, $N$~is zero-dimensional.
Thus,~the orbit of each
\mbox{$x \hspace{-2pt} \in \hspace{-2pt} N\hspace{-2pt}$}
is a singleton. Therefore, 
\mbox{$g^{-1} x g \hspace{-2pt} = \hspace{-2pt} x$} for every 
\mbox{$g \hspace{-2pt}\in\hspace{-2pt} (\widetilde G)_0$}
and \mbox{$x \hspace{-2pt} \in \hspace{-2pt} N\hspace{-2pt}$}. In 
particular, the elements of $H$ and $N$ commute (elementwise).
We note that this argument, concerning the commuting of connected and 
zero-dimensional normal subgroups, is due to K. H. Hofmann 
(cf.~\cite{Hof0dim-trick}).

Let \mbox{$f\colon \hspace{-1.25pt} \mathbb{R} \hspace{-1pt} 
\rightarrow\hspace{-1.5pt} H$} be a topological isomorphism. 
The continuous surjection \mbox{$h\colon \hspace{-1.25pt} 
N \hspace{-2.25pt}\times\hspace{-0.75pt} \mathbb{R} 
\hspace{-0.5pt}\rightarrow \hspace{-2pt} N \hspace{-0.6pt}H$}~given by 
\mbox{$h(x,r)\hspace{-2pt} =\hspace{-2pt}xf(r)$} is a  homomorphism, 
because  $f$ is a homomorphism, and $N$ and $H$  \mbox{commute} 
(elementwise). We show that $h$ is a topological  isomorphism. Since 
\mbox{$H\hspace{-2pt} \cong \hspace{-2pt} \mathbb{R}$}, the only compact 
subgroup of $H$ is the trivial one, and thus 
\mbox{$N\hspace{-2.5pt}\cap\hspace{-2pt}H\hspace{-2pt}
= \hspace{-2pt} \{e\}$}. 
Therefore, $h$ is injective. Since~$N$ is compact and normal, and $H$ is 
closed in $\widetilde  G\hspace{-1pt}$,  the subgroup $NH$ is closed in 
$\widetilde G\hspace{-1pt}$, and so $NH$ is locally compact. The domain
\mbox{$N \hspace{-2pt}\times\hspace{-0.5pt} \mathbb{R}$} of $h$ is 
also locally compact and $\sigma$-compact. Consequently, by the Open 
Mapping Theorem, $h$ is open  (cf.~\cite[5.29]{HewRos}). Hence, 
$h$ is a topological isomorphism.

The subgroup $NH$~is $G_\delta\mbox{-}$open in~$\widetilde 
G\hspace{-1pt}$, because it  contains the $G_\delta$-set $N\hspace{-2pt}$. 
By  Theorem~\ref{prel:thm:lcps}, $G$ is $G_\delta$-dense in $\widetilde 
G\hspace{-1pt}$.  By Theorem~\ref{prel:thm:delta}(a), the 
$G_\delta\mbox{-}$topology  is  a group topology on~$\widetilde 
G\hspace{-1pt}$,~and so  Lemma~\ref{herLPS:lemma:dense} is applicable. 
Consequently, by Lemma~\ref{herLPS:lemma:dense},
\mbox{$P \hspace{-2pt}:= \hspace{-2pt} NH \hspace{-2pt}\cap\hspace{-2pt} 
G$} is $G_\delta$-dense in $NH\hspace{-0.5pt}$; in particular, $P$ is 
dense in~$NH\hspace{-1pt}$, and by Theorem~\ref{prel:thm:comp},
\mbox{$\widetilde P \hspace{-2pt} = \hspace{-2pt} NH\hspace{-1pt}$}.
Since $P$ is a closed subgroup of $G\hspace{-1pt}$, it is
hereditarily locally pseudocompact. Thus,
\mbox{$P^\prime \hspace{-2pt}:= \hspace{-2pt} h^{-1}(P)$} 
satisfies the conditions of Proposition~\ref{herLPS:prop:NxR}. 
Consequently,
\mbox{$H \hspace{-2pt} = \hspace{-1pt}
h(\{e\}\hspace{-2.5pt}\times \hspace{-1.6pt} \mathbb{R})
\hspace{-2pt}\subseteq \hspace{-2pt}h(P^\prime)\hspace{-2pt}
\subseteq \hspace{-2pt} G$\hspace{-1pt}}. Hence, $G_0$ is non-trivial, as 
desired.
\end{proof}

\section{Proof of Theorems~\ref{thm:main:coarser} and 
\ref{thm:main:totminLPS}}

\label{sect:short}

Recall that a group topology is {\itshape linear} if it admits a base at 
the identity consisting of subgroups. Since every open subgroup is also 
closed (cf.~\cite[5.5]{HewRos} and~\cite[1.10(c)]{GLCLTG}), every linear 
group topology is zero-dimensional. We prove a slightly stronger version 
of Theorem~\ref{thm:main:coarser}.

\begin{Ltheorem*}[\ref{thm:main:coarser}$'$] \mbox{ }

\begin{myalphlist}

\item
Every locally pseudocompact totally disconnected group admits a coarser
linear group topology.

\item
Every minimal, locally pseudocompact, totally disconnected group  
has a linear topology, and thus it is strongly Vedenissoff.

\end{myalphlist}
\end{Ltheorem*}

\begin{proof}
(a) Since $G$ is totally disconnected,  
\mbox{$q(G) \hspace{-2pt} = \hspace{-2pt} \{e\}$}, and since $G$ is 
locally pseudocompact,~by Theorem~\ref{prel:thm:connsum}(a), 
\mbox{$q(G) \hspace{-2pt} = \hspace{-2pt} o(G)$}. Thus,
\mbox{$o(G) \hspace{-2pt} = \hspace{-2pt} \{e\}$}. Therefore,
the family of open subgroups in $G$ forms a base at the identity 
for a (Hausdorff) group topology on $G\hspace{-1pt}$, and it is obviously 
coarser than the topology of $G\hspace{-1pt}$. Clearly, this topology is 
linear.

(b) follows from (a) and the definition of minimality.
\end{proof}

\begin{Ltheorem*}[\ref{thm:main:totminLPS}]
Let $G$ be a totally minimal locally pseudocompact group. Then  
\mbox{$G_0 \hspace{-2pt} = \hspace{-2pt} q(G)$} if and only if
$G/G_0$ is zero-dimensional, in which case
$G$ is strongly Vedenissoff.
\end{Ltheorem*}

\begin{proof}
Suppose that \mbox{$G_0 \hspace{-2pt} = \hspace{-2pt} q(G)$}.
Then the quotient \mbox{$G/G_0 \hspace{-2pt} = \hspace{-2pt} G/q(G)$} is 
minimal, locally pseudocompact, and totally disconnected. Thus, by
Theorem~\ref{thm:main:coarser}(b), $G/G_0$ is zero-dimensional. By 
Theorem~\ref{prel:thm:connsum}(a),
\mbox{$q(G) \hspace{-2pt} = \hspace{-2pt} o(G)$}, and in particular,
\mbox{$z(G) \hspace{-2pt} = \hspace{-2pt} o(G)$}. Therefore, $G$ is 
strongly Vedenissoff, as required. 
The converse follows by Theorem~\ref{prel:thm:connsum}(c).
\end{proof}

\section{Proof of Theorem~\ref{thm:main:example}}

\label{sect:exmp}

\begin{Ltheorem*}[\ref{thm:main:example}]
For every natural number $n$ or \mbox{$n\hspace{-2pt} =\hspace{-1pt} 
\omega$}, there exists an abelian pseudocompact group $G_n$ such that 
$G_n$ is perfectly totally minimal, hereditarily disconnected, but
\mbox{$\dim G_n \hspace{-2pt} = \hspace{-1pt}n$}.
\end{Ltheorem*}

In this section, we prove Theorem~\ref{thm:main:example} by establishing a 
general construction that allows one to ``realize" minimal abelian groups 
as quasi-components of minimal pseudocompact groups. A weaker version of 
Theorem~\ref{thm:main:example}, which provides totally minimal 
pseudocompact groups, was announced in~\cite[1.4.2]{Dikconcomp}.  The 
novelty of Theorem~\ref{thm:main:example}, in addition to its complete 
proof, is that we obtain {\itshape perfectly} totally minimal 
pseudocompact groups.

\begin{Ltheorem*}[\ref{thm:main:example}$'$]
Let $A$ be a precompact abelian group that is contained in a connected 
compact abelian group $C\hspace{-1pt}$. Then there exists a 
pseudocompact abelian group $G$ such that 
\mbox{$A \hspace{-2pt}\cong \hspace{-1.7pt} q(G)$} and
\mbox{$C \hspace{-2pt} \cong \hspace{-2pt}  (\widetilde G)_0$}, and in 
particular,
\mbox{$\dim G \hspace{-1.5pt} = \hspace{-1pt} \dim C\hspace{-1pt}$}.
Furthermore, if \mbox{$C\hspace{-2pt} = \hspace{-2pt} \widetilde A$} and

\begin{myalphlist}

\item
$A$ is minimal, then $G$ may be chosen to be  minimal;

\item
$A$ is totally minimal, then $G$ may be chosen to be totally 
minimal;

\item
$A$ is perfectly minimal, then $G$ may be chosen to be 
perfectly minimal;

\item
$A$ is perfectly totally minimal, then $G$ may be chosen to 
be perfectly totally minimal.

\end{myalphlist}
\end{Ltheorem*}

Theorem~\ref{thm:main:example}$'$ follows a line of ``embedding" results, 
which state that certain (locally) precompact groups embed into (locally) 
pseudocompact groups as a particular (e.g., functorial) closed subgroup 
(cf.~\cite[2.1]{ComfvMill1}, \cite{Ursul1}, \cite[7.6]{ComfvMill2}, 
\cite{Ursul2}, \cite[3.6]{Dikconpsc}, and \cite[5.6]{ComfGL}). The novelty 
is that  minimality properties of the group $A$ are inherited by the group 
$G$ that is constructed. By the celebrated
Prodanov-Stoyanov  Theorem, every minimal abelian group is precompact
(cf.~\cite{ProdStoj} and~\cite{ProdStoj2}), and so  the condition that 
the group $A$ is precompact is not restrictive at all.

We first show how Theorem~\ref{thm:main:example} follows from 
Theorem~\ref{thm:main:example}$'$, and then proceed to proving the latter. 
To that end, we recall a characterization due to Stoyanov for groups that 
are not only perfectly totally minimal, but their powers have the same 
property too (cf.~\cite{Stoy}).
For an abelian topological group $G\hspace{-1pt}$, let $wtd(G)$ 
denote the subgroup of elements $x$ in $G$ for which there exists a 
positive integer $m$ such that for every sequence $\{k_n\}_{n=1}^\infty$ 
of integers, one has \mbox{$m^n k_n x \longrightarrow 0$} in $G\hspace{-1pt}$. 
In other words,
\begin{align*}
wtd(G):=\{ x \in G \mid \exists m > 0, \forall \{k_n\}_{n=1}^\infty \in 
\mathbb{N}^\omega, m^n k_n x 
\longrightarrow 0\}.
\end{align*}

\begin{ftheorem}[{\cite{Stoy}, \cite[6.1.18]{DikProSto}}] 
\label{exmp:thm:wtd}
Let $P$ be a precompact abelian group. Then $P^\lambda$ is perfectly 
totally minimal for every cardinal $\lambda$ if and only if 
\mbox{$wtd(\widetilde P)\hspace{-2pt} \subseteq \hspace{-2pt} 
P\hspace{-1pt}$}.
\end{ftheorem}

\begin{proof}[Proof of Theorem~\ref{thm:main:example}.]
Put \mbox{$P\hspace{-2pt} := \hspace{-2pt} \mathbb{Q}/\mathbb{Z}$}. Then
\mbox{$\widetilde P\hspace{-2pt} = \hspace{-2pt} \mathbb{R}/\mathbb{Z}$} 
and 
\mbox{$wtd(\widetilde P)\hspace{-2pt} = \hspace{-2pt} \mathbb{Q}/\mathbb{Z} 
\hspace{-2pt} = \hspace{-2pt} P$}, and by Theorem~\ref{exmp:thm:wtd},
\mbox{$A_n\hspace{-2pt} := \hspace{-2pt} P^n$} is perfectly totally minimal 
for every natural number $n$ or 
\mbox{$n\hspace{-2pt} =\hspace{-1pt}\omega$}, and $A$ is contained in the 
connected compact group 
\mbox{$C_n\hspace{-2pt} := \hspace{-2pt} (\mathbb{R}/\mathbb{Z})^n$}.
By Theorem~\ref{thm:main:example}$'$(d), there exists
a~perfectly totally minimal pseudocompact group $G_n$ such that 
\mbox{$A_n \hspace{-2pt} \cong \hspace{-2pt} q(G_n)$} and 
\mbox{$\dim G_n\hspace{-2pt} = \hspace{-1pt} \dim C_n 
\hspace{-2pt} = \hspace{-1pt} n$}. Since
\mbox{$(G_n)_0 \hspace{-2pt} \subseteq \hspace{-2pt} q(G_n)_0
\hspace{-2pt}\cong \hspace{-2pt} (A_n)_0 
\hspace{-2pt} = \hspace{-2pt} \{0\}$}, the group $G_n$ is hereditarily 
disconnected, as desired.
\end{proof}

We proceed now to proving Theorem~\ref{thm:main:example}$'$. The proof 
has two ingredients: A zero-dimensional pseudocompact group $H$ 
with good minimality properties, and a discontinuous homomorphism
\mbox{$h\colon \widetilde H \rightarrow C$} with kernel $H\hspace{-1pt}$. 
The desired group $G$ will be the sum the of graph of $h$ and the group 
$A$ formed in the product 
\mbox{$H\hspace{-2.5pt} \times \hspace{-2.25pt} C\hspace{-1.5pt}$}.

\begin{lemma} \label{exmp:lemma:kernel}
For every infinite cardinal $\lambda$, there exists a 
pseudocompact zero-dimensional group $H$ such that:

\begin{myromanlist}

\item
$H$ is perfectly totally minimal;

\item
\mbox{$r_0(\widetilde H / H) \hspace{-2pt} \geq 
\hspace{-1pt}  2^\lambda$}.
\end{myromanlist}

\end{lemma}

\begin{proof}
Let $\mathbb{P}$ denote the set of prime integers, and for 
\mbox{$p\hspace{-2pt} \in \hspace{-2pt} \mathbb{P}$}, let
$\mathbb{Z}_p$ denote the group of $p$-adic integers.
Put \mbox{$N \hspace{-2pt} := \hspace{-2pt}
\prod\limits_{p\in \mathbb{P}} \hspace{-1pt}
\mathbb{Z}_p^{\omega_1}$}. We think of elements of $N$ as tuples
\mbox{$(x_{p,\alpha})$}, where 
\mbox{$p\hspace{-2pt}\in\hspace{-2pt}\mathbb{P}$} and 
\mbox{$\alpha \hspace{-2pt}<\hspace{-2pt} \omega_1$}.~We define three 
subgroups of $N$:

\begin{mynumlist}

\item
\mbox{$E_1\hspace{-2pt} :=\hspace{-2pt}
\bigoplus\limits_{p\in \mathbb{P}} 
\hspace{-1pt}\mathbb{Z}_p^{\omega_1}\hspace{-1.5pt}$} consists of 
elements $x$ such that  
\mbox{$(\exists \alpha)(x_{p,\alpha}\hspace{-2pt}\neq\hspace{-2pt} 0)$} 
only for finitely many primes $p$ (or equivalently,
\mbox{$E_1\hspace{-2pt} = \hspace{-2pt} wtd(N)$});


\item
\mbox{$E_2\hspace{-2pt} :=
\hspace{-2pt} \prod\limits_{p\in \mathbb{P}}
\hspace{-2pt} S_p$}, where  $S_p$ is the $\Sigma$-product of 
$\omega_1$-many copies of $\mathbb{Z}_p$ 
(or equivalently, $E_2$ 
consists of elements  $x$ such that all but countably many coordinates 
$x_{p,\alpha}$ of $x$ are zero);

\item
\mbox{$E \hspace{-2pt} := \hspace{-2pt} 
E_1 \hspace{-1pt} + \hspace{-1pt} E_2$}.

\vspace{6pt}

\end{mynumlist}
We claim that \mbox{$H \hspace{-2pt}: = \hspace{-2pt} E^\lambda$} has the
desired properties.

The group $E$ is $G_\delta$-dense in $N\hspace{-1.5pt}$, because it 
contains $E_2$, which is clearly $G_\delta$-dense. Thus, $H$~is 
$G_\delta$-dense in the compact group $N^\lambda\hspace{-1.5pt}$, and 
in particular, by Theorem~\ref{exmp:thm:wtd}, 
\mbox{$\widetilde H \hspace{-2pt} = \hspace{-2pt} 
N^\lambda\hspace{-1.5pt}$}. Therefore, by Theorem~\ref{prel:thm:lcps},
$H$ is pseudocompact.  By Theorem~\ref{prel:thm:connsum}(b), $H$ is 
zero-dimensional, because  \mbox{$N^\lambda\hspace{-0.5pt}$} is so.

Since $E$ is $G_\delta$-dense in the compact group $N\hspace{-1.5pt}$,
in particular, it is dense, and by  Theorem~\ref{prel:thm:comp},
\mbox{$N\hspace{-2pt} = \hspace{-2pt} \widetilde E$}. Thus,
\mbox{$wtd(\widetilde E)\hspace{-2pt} = \hspace{-2pt}
wtd(N) \hspace{-2pt} = \hspace{-2pt} E_1 
\hspace{-2pt} \subseteq \hspace{-2pt} E$}, and therefore by 
Theorem~\ref{exmp:thm:wtd}, 
\mbox{$H\hspace{-2pt} = \hspace{-2pt} E^\lambda$} is perfectly totally 
minimal.

In order to prove that 
\mbox{$r_0(\widetilde H / H) \hspace{-2pt} \geq
\hspace{-1pt}  2^\lambda$}, it suffices to show that
\mbox{$r_0(N / E) \hspace{-2pt} \geq \hspace{-1pt}  1$}, as
\mbox{$\widetilde H / H \hspace{-2pt}
\cong \hspace{-2pt} (N/E)^\lambda$}.
Let $\Delta$ denote the ``diagonal" subgroup of $N\hspace{-1.5pt}$, that 
is, the subgroup generated by $d$ such that 
\mbox{$d_{p,\alpha}\hspace{-2pt} = \hspace{-2pt}1$} for every $p$ and 
$\alpha$, and we prove that \mbox{$E\hspace{-2pt} \cap \hspace{-2pt} \Delta 
\hspace{-2pt} = \hspace{-2pt} \{0\}$}. In fact, we show a bit more, 
namely, that every element in $E$ has at least one zero 
coordinate. Let 
\mbox{$x\hspace{-2pt}=\hspace{-2pt} y 
\hspace{-1pt}+ \hspace{-1pt} z \hspace{-2pt} \in \hspace{-2pt} E$}, 
where \mbox{$y \hspace{-2pt}\in\hspace{-2pt} E_1$} and
\mbox{$z \hspace{-2pt}\in\hspace{-2pt} E_2$}. By the definition of $E_1$, 
there exists \mbox{$q \hspace{-2pt}\in\hspace{-2pt} \mathbb{P}$} such that 
\mbox{$y_{q,\alpha}\hspace{-2pt} = \hspace{-2pt} 0$} for every 
\mbox{$\alpha\hspace{-2pt} < \hspace{-2pt}\omega_1$}. 
Since \mbox{$z \hspace{-2pt}\in\hspace{-2pt} E_2$}, all but countably 
many 
coordinates of $z$ are zero. In particular, there exists 
\mbox{$\gamma\hspace{-2pt} < \hspace{-2pt}\omega_1$} such that 
\mbox{$z_{q,\gamma}\hspace{-2pt} = \hspace{-2pt} 0$}. Therefore,
\mbox{$x_{q,\gamma}\hspace{-2pt}=\hspace{-2pt} y_{q,\gamma}
\hspace{-1pt}+ \hspace{-1pt} z_{q,\gamma} 
\hspace{-2pt} = \hspace{-2pt} 0$}. Hence,
\mbox{$E\hspace{-2pt} \cap \hspace{-2pt}\Delta
\hspace{-2pt} = \hspace{-2pt} \{0\}$} and
\mbox{$r_0(N / E) \hspace{-2pt} \geq \hspace{-1pt}  1$}, as desired.
\end{proof}

We consider the next lemma part of the folklore of pseudocompact 
abelian groups (cf.~\cite[3.6, 3.10]{ComfGal2}), and we provide its proof 
only for the sake of completeness.

\begin{flemma} \label{exmp:lemma:graph}
Let $K_1$ and $K_2$ be compact topological groups, and let
\mbox{$h\colon K_1 \hspace{-2pt}\rightarrow\hspace{-2pt} K_2$} be a 
surjective homomorphism such that $\ker h$ is $G_\delta$-dense in $K_1$. 
Then the graph $\Gamma_h$ of $h$ is a $G_\delta$-dense 
subgroup of the product \mbox{$K_1\hspace{-2pt} \times \hspace{-2.25pt} K_2$},
and in particular, $\Gamma_h$ is pseudocompact.
\end{flemma}

\begin{proof}
Let $B$ be a non-empty $G_\delta$-subset of 
\mbox{$K_1\hspace{-2pt} \times \hspace{-2.25pt} K_2$}. Without loss of 
generality, we may assume that $B$ is of the form
\mbox{$B_1\hspace{-2pt} \times \hspace{-2.25pt} B_2$}, where $B_i$ is a 
$G_\delta$-set in $K_i$. Pick 
\mbox{$x_2 \hspace{-2pt} \in \hspace{-2pt} B_2$}. Since $h$ is surjective, 
there is \mbox{$x_1^\prime\hspace{-2pt} \in \hspace{-2pt} K_1$} such that
\mbox{$h(x_1)\hspace{-2pt}=\hspace{-2pt}x_2$}. The translate
$B_1x_1^{-1}$ is a non-empty $G_\delta$-set in $K_1$, and thus we may pick
\mbox{$x_0 \hspace{-2pt} \in \hspace{-2pt} 
B_1x_1^{-1} \hspace{-2pt} \cap\hspace{-1pt} \ker h$}, because
$\ker h$ is $G_\delta$-dense in $K_1$. Since
\mbox{$h(x_0x_1)\hspace{-2pt} = \hspace{-2pt}h(x_1)
\hspace{-2pt} = \hspace{-2pt}x_2$}, one obtains that
\mbox{$(x_0 x_1,x_2) \hspace{-2pt} \in \hspace{-2pt}
\Gamma_h \hspace{-2pt}\cap \hspace{-2pt}
(B_1\hspace{-2pt} \times \hspace{-2.25pt} B_2)$}. This shows that 
$\Gamma_h$ meets every $G_\delta$-set in 
\mbox{$K_1\hspace{-2pt} \times \hspace{-2.25pt} K_2$}. Therefore, by 
Theorem~\ref{prel:thm:comp}, 
\mbox{$K_1\hspace{-2pt} \times \hspace{-2.25pt} K_2$} is the completion of 
$\Gamma_h$. Hence, by Theorem~\ref{prel:thm:lcps}, $\Gamma_h$ is 
pseudocompact.
\end{proof}

A last, auxiliary tool in the proof of Theorem~\ref{thm:main:example}$'$ 
is the following observation.

\begin{remark} \label{exmpl:rem:dense}
Let $\mathcal{P}$ denote one of the following properties: minimal, totally 
minimal, perfectly minimal, perfectly totally minimal. If $G$ contains a 
dense subgroup with property $\mathcal{P}\hspace{-1.5pt}$, then $G$  also 
has property $\mathcal{P}\hspace{-2pt}$ 
(cf.~\cite[Theorem~2]{Steph}, \cite{Prod1}, 
\cite[Propositions~1~and~2]{Banasch}, \cite{DikPro}, 
\cite[2.5.1, 4.3.3]{DikProSto}, and \cite[3.21, 3.23]{GLCLTG}). 
\end{remark}

\begin{proof}[Proof of Theorem~\ref{thm:main:example}$'$.]
Put \mbox{$\lambda \hspace{-2pt} = \hspace{-2pt} w(C)$}, and let $H$ be 
the 
group provided by Lemma~\ref{exmp:lemma:kernel}. Since
\mbox{$r_0(\widetilde H / H) \hspace{-2pt} \geq \hspace{-1pt}  2^\lambda$},
the quotient \mbox{$\widetilde H / H$} contains a free abelian group 
$F$ of rank $2^\lambda$. 
As \mbox{$|C|\hspace{-2pt} \leq \hspace{-2pt}2^{\lambda}$}, one may
pick a surjective homomorphism \mbox{$h_1\colon F \rightarrow C\hspace{-1pt}$}.
The group $C$ is divisible, because it is compact and connected 
(cf.~\cite[24.25]{HewRos}). Thus, $h_1$ can be extended to a surjective 
homomorphism
 \mbox{$h_2\colon \widetilde H / H \rightarrow C\hspace{-1pt}$}. 

Let \mbox{$h\colon \widetilde H \rightarrow C\hspace{-1pt}$} denote 
the composition of $h_2$ with the canonical projection
\mbox{$\widetilde H  \rightarrow \widetilde H / H\hspace{-1pt}$}.
By Theorem~\ref{prel:thm:lcps}, $H$ is $G_\delta$-dense in 
$\widetilde H\hspace{-1.5pt}$, 
because $H$ is pseudocompact. Thus, $\ker h$ is $G_\delta$-dense in 
\mbox{$\widetilde H\hspace{-1.5pt}$}, because  
\mbox{$H \hspace{-2pt} \subseteq \hspace{-2pt} \ker h$}. Clearly, $h$ is 
surjective. Therefore, by Lemma~\ref{exmp:lemma:graph}, the graph
$\Gamma_h$ of $h$ is $G_\delta\mbox{-}$dense in the product
\mbox{$\widetilde H\hspace{-2pt}\times \hspace{-1.5pt} C\hspace{-1pt}$}.

Put \mbox{$G \hspace{-2pt} := \hspace{-2pt} 
\Gamma_h \hspace{-1pt}+ \hspace{-1pt}(\{0\} 
\hspace{-2pt}\times\hspace{-2pt} A)$}. Since $\Gamma_h$ is 
$G_\delta$-dense in \mbox{$H\hspace{-2pt}\times \hspace{-1.5pt} C$} and 
contained in $G\hspace{-1pt}$, the group $G$ is $G_\delta\mbox{-}$dense 
too.  Thus,  by Theorem~\ref{prel:thm:comp}, 
\mbox{$\widetilde G \hspace{-2pt} = \hspace{-2pt}
\widetilde H\hspace{-2pt}\times \hspace{-1.5pt} C\hspace{-1pt}$}, and  by 
Theorem~\ref{prel:thm:lcps}, $G$ is pseudocompact. As $H$ is 
zero-dimensional, 
\mbox{$(\widetilde G)_0 \hspace{-2pt} = \hspace{-2pt} 
\{0\} \hspace{-2pt} \times \hspace{-2pt} C\hspace{-1pt}$}, and by 
Theorem~\ref{prel:thm:connsum}(a), 
\mbox{$q(G)\hspace{-2pt} = 
\hspace{-2pt} (\widetilde G)_0 \hspace{-2.1pt} \cap\hspace{-2pt} G
\hspace{-2pt} = \hspace{-2pt} 
\{0\} \hspace{-2pt} \times \hspace{-2pt} A$}. 

We check now that \mbox{$\dim G \hspace{-1.5pt} = 
\hspace{-1pt} \dim C\hspace{-1pt}$}. Since $G$ is pseudocompact, by 
Theorem~\ref{prel:thm:lcps}, 
\mbox{$\widetilde G \hspace{-2pt} = \hspace{-2pt} \beta 
G\hspace{-1pt}$},~and so 
\mbox{$\dim G \hspace{-1.5pt} = \hspace{-1pt} \dim \beta G  
 \hspace{-1.5pt} = \hspace{-1pt} \dim \widetilde G\hspace{-1pt}$} 
(cf.~\cite[7.1.17]{Engel6}).
As $H$ is zero-dimensional and pseudocompact, by 
Theorem~\ref{prel:thm:connsum}(b), 
\mbox{$\dim\widetilde  H \hspace{-2pt} = \hspace{-1pt} 0$}.
Thus, by Yamanoshita's Theorem, 
\mbox{$\dim \widetilde G \hspace{-1.5pt} = \hspace{-1pt}
\dim \widetilde H \hspace{-1pt} + \hspace{-1pt} \dim C 
\hspace{-1.5pt} = \hspace{-1pt}  \dim C$}
(cf.~\cite{Yamanoshita}, \cite[Corollary~2]{MostertDim},
and~\cite[3.3.12]{DikProSto}). Therefore,  
\mbox{$\dim G \hspace{-1.5pt} =\hspace{-1pt} \dim C\hspace{-1pt}$}.

We turn to minimality properties of $G\hspace{-1pt}$. 
Suppose that  \mbox{$C\hspace{-2pt} = \hspace{-2pt} \widetilde A$}.
The group $G$ always 
contains the product \mbox{$H\hspace{-2pt}\times \hspace{-2.5pt} A$}, but
in this case, 
\mbox{$H\hspace{-2pt}\times \hspace{-2.5pt} A$} is dense in 
\mbox{$\widetilde G \hspace{-2pt} = \hspace{-2pt}
\widetilde H\hspace{-2pt}\times \hspace{-1.5pt} C\hspace{-1pt}$}, and thus 
it is dense in $G\hspace{-1pt}$. Therefore,~by 
Remark~\ref{exmpl:rem:dense}, $G$ inherits all minimality properties of 
\mbox{$H\hspace{-2pt}\times \hspace{-2.5pt} A$}. Since $H$ is perfectly 
totally minimal, the product 
\mbox{$H\hspace{-2pt}\times \hspace{-2.5pt} A$} inherits all minimality 
properties of $A$. This shows (a)-(d).
\end{proof}

One wonders whether the condition 
\mbox{$C\hspace{-2pt} = \hspace{-2pt} \widetilde A$} is necessary 
for parts (a)-(d) of Theorem~\ref{thm:main:example}$'$. If the resulting 
group $G$ is to be totally minimal, then the answer is positive. Dikranjan 
showed that if $G$ is a minimal pseudocompact abelian group then $q(G)$ is 
dense in $(\widetilde G)_0$  if and only if $G/q(G)$ is minimal 
(cf.~\cite[1.7]{DikPS0dim}), in which case $(\widetilde G)_0$ is the 
completion of $q(G)$. This settles the question for (b) and (d).
The following remark settles the question for (a) and (c).

\begin{remark}
We note (without a proof) that 
the techniques of Theorem~\ref{thm:main:example}$'$ can also be used to 
construct, for every positive integer $n$ or 
\mbox{$n\hspace{-1pt}=\hspace{-1pt} \omega$}, a perfectly minimal 
pseudocompact $n$-dimensional group $G$ such that 
$G/q(G)$ is not minimal, and hence 
$q(G)$ is not dense in $(\widetilde G)_0$.
\end{remark}

\section{Concluding remarks}

One can also define the intersection $o^*(G)$ of all open {\itshape 
normal} subgroups of a group $G\hspace{-1pt}$, and ask about its 
relationship with the other four functorial subgroups. If a locally 
compact group $L$ admits a base at the identity consisting of 
neighborhoods that are invariant under conjugation (that is, $L$~is 
so-called {\itshape balanced} or admits {\itshape Small Invariant 
Neighborhoods}), which is the case for compact or abelian groups, then
\mbox{$o(L) \hspace{-2pt} = \hspace{-2pt} o^*(L)$}. There are, however, 
many locally compact groups that do not have this property.

\begin{examples} \mbox{ }

\begin{myalphlist}

\item
The semidirect product \mbox{$L \hspace{-2pt} :=\hspace{-2pt}
\{0,1\}^\mathbb{Z} \hspace{-2.5pt} \rtimes\hspace{-1.5pt} \mathbb{Z}$}, 
where  $\mathbb{Z}$ acts on the compact group 
\mbox{$K \hspace{-2pt} := \hspace{-2pt} \{0,1\}^\mathbb{Z}$} by 
shifts, is locally compact and zero-dimensional, and thus 
$o(L)$ is trivial. However, $K$ is the smallest open normal subgroup of 
$L$, and therefore \mbox{$o^*(L)\hspace{-2pt} = \hspace{-2pt} 
K\hspace{-1pt}$}.

\item
For \mbox{$p\hspace{-2pt} \in \hspace{-2pt} \mathbb{P}$}, let
$\mathbb{Q}_p$ denote the (locally compact) field of $p$-adic numbers. The 
discrete multiplicative group $\mathbb{Q}^\times$ of non-zero rationals 
acts on 
$\mathbb{Q}_p$ by multiplication. The semidirect product
\mbox{$L\hspace{-2pt} := \hspace{-2pt}
(\mathbb{Q}_p,+) \hspace{-2pt} \rtimes 
\hspace{-1.5pt} \mathbb{Q}^\times\hspace{-1pt}$}
is locally compact and zero-dimensional (and so, again, $o(L)$ is 
trivial), but $\mathbb{Q}_p$ is the smallest open normal subgroup of $L$,
and therefore  
\mbox{$o^*(L)\hspace{-2pt} = \hspace{-2pt} \mathbb{Q}_p$}.

\item
In general, let $G$ be a locally compact group and 
$D$ a subgroup of $\operatorname{Aut}(G)$ such that $G$ 
contains no proper $D$-invariant open subgroup, and put
\mbox{$L\hspace{-2pt} := \hspace{-2pt}
G \hspace{-2pt} \rtimes \hspace{-2pt} D$},  where $D$ is equipped 
with the  discrete  topology. Then, by Theorem~\ref{intro:thm:LC}, the 
locally  compact  group $L$ has the property that
\mbox{$o(L)\hspace{-2pt} = \hspace{-2pt} L_0
\hspace{-2pt} = \hspace{-2pt} G_0$} and
\mbox{$o^* (L) \hspace{-2pt} = \hspace{-2pt} G$}.

\end{myalphlist}
\end{examples}

\bigskip

\section*{Acknowledgements}

This work has emerged from the joint work of one of the authors with W. W. 
Comfort on locally precompact groups; the authors wish to express their 
heartfelt gratitude to Wis Comfort for the helpful discussions and 
correspondence. The authors wish to thank Dragomir Djokovic for the 
valuable correspondence. The authors are grateful to Karen Kipper for her 
kind help in proofreading this paper for grammar and punctuation.

{\footnotesize

\bibliography{notes,notes2,notes3}
}

\begin{samepage}

\bigskip
\noindent
\begin{tabular}{l @{\hspace{1.8cm}} l}
Department of Mathematics and Computer Science & Department of 
Mathematics\\
University of Udine & University of Manitoba\\
Via delle Scienze, 208 -- Loc. Rizzi, 33100 Udine
 & Winnipeg, Manitoba, R3T 2N2 \\
Italy & Canada \\ & \\
\em e-mail: dikranja@dimi.uniud.it  &
\em e-mail: lukacs@cc.umanitoba.ca
\end{tabular}

\end{samepage}

\end{document}